\documentclass[a4paper]{article}

\usepackage[english]{babel}
\usepackage[letterpaper,top=2cm,bottom=2cm,left=25mm,right=25mm,marginparwidth=2cm]{geometry}
\usepackage{amsmath}
\usepackage{amssymb}
\usepackage{amsthm}
\usepackage{thmtools}
\usepackage{bbm}
\usepackage{mathtools}
\usepackage{enumitem}
\usepackage{xparse}
\usepackage{siunitx}
\usepackage{subcaption}
\usepackage{todonotes}

\newtheorem{basethm}{BaseTheorem}
\newtheorem{lemma}[basethm]{Lemma}
\newtheorem{theorem}[basethm]{Theorem}
\newtheorem{proposition}[basethm]{Proposition}
\newtheorem{definition}[basethm]{Definition}

\usepackage[colorlinks=true, allcolors=blue]{hyperref}
\usepackage[noabbrev,capitalize]{cleveref}

\DeclareDocumentCommand\zerovec{o}{\IfNoValueTF{#1}{\mathbb{O}}{\mathbb{O}_{#1}}}

\DeclareDocumentCommand\R{}{\mathbb{R}}

\DeclareDocumentCommand\rank{o}{\operatorname{rank}\IfValueTF{#1}{\left(#1\right)}{}}

\DeclareDocumentCommand\orderO{o}{\mathcal{O}\IfValueTF{#1}{\ensuremath{\left(#1\right)}{}}}
\DeclareDocumentCommand\orderTheta{o}{\ensuremath{\Theta\IfValueTF{#1}{\left(#1\right)}}}
 	
\DeclareDocumentCommand\proj{oo}{\IfValueTF{#1}{\operatorname{proj}{}_{#1}}{%
  \operatorname{proj}{}}\IfValueTF{#2}{\left(#2\right)}{}}

\DeclareDocumentCommand\D{}{\mathcal{D}}
\DeclareDocumentCommand\V{}{\mathcal{V}}
\DeclareDocumentCommand\A{}{\mathcal{A}}
\DeclareDocumentCommand\E{}{\mathcal{E}}
\DeclareDocumentCommand\targets{}{\mathcal{T}}
\DeclareDocumentCommand\singletons{}{\mathcal{S}}
\DeclareDocumentCommand\inNeighbors{}{N^{\text{in}}}
\DeclareDocumentCommand\outNeighbors{}{N^{\text{out}}}
\DeclareDocumentCommand\flowerRel{}{\operatorname{FR}}

\DeclareDocumentCommand\ML{}{\operatorname{ML}}
\DeclareDocumentCommand\conv{o}{\operatorname{conv}\IfValueTF{#1}{\left(#1\right)}{}}

\author{Emily Schutte\footnote{University of Luxembourg, Luxembourg. E-mail: emilyschutte@live.nl} \and Matthias Walter\footnote{Department of Applied Mathematics, University of Twente, The Netherlands. E-mail: m.walter@utwente.nl} \thanks{The second author acknowledges funding support from the Dutch Research Council (NWO) on grant number OCENW.M20.151.}}

\title{Relaxation strength for multilinear optimization: McCormick strikes back}

\begin{document}

\maketitle

\begin{abstract}
  We consider linear relaxations for multilinear optimization problems.
  In a recent paper, Khajavirad proved that the extended flower relaxation is at least as strong as the relaxation of any recursive McCormick linearization (Operations Research Letters 51 (2023) 146--152).
  In this paper we extend the result to more general linearizations, and present a simpler proof.
  Moreover, we complement Khajavirad's result by showing that the intersection of the relaxations of such linearizations and the extended flower relaxation are equally strong.
\end{abstract}

\section{Introduction}
\label{sec_intro}

\DeclareDocumentCommand\I{}{\mathcal{I}}

We consider multilinear optimization problems
\begin{subequations}
  \label{model_multilinear}
  \begin{alignat}{7}
    & \text{min }
      & \sum_{I \in \I_0} c^0_I \prod_{v \in I} x_v \\
    & \text{s.t. }
      & \sum_{I \in \I_j} c^j_I \prod_{v \in I} x_v &\leq b_j &\quad& \forall j \in \{1,2,\dotsc,m\} \\
    & & x_v &\in [\ell_v, u_v] &\quad& \forall v \in V,
  \end{alignat}
\end{subequations}
where $V$ denotes the variables and $\I_0,\I_1,\I_2,\dotsc,\I_m \subseteq V$ are families of subsets thereof, where $c^j_I \in \R$ and $b_j \in \R$ are the coefficients of the \emph{monomials}, and the right-hand side, respectively, and where $\ell, u \in \R^V$ are the bounds on the variables.
Note that every constraint (and the objective function) is a multilinear polynomial, which means that every variable has an exponent equal to either $0$ or $1$ in each monomial.
We refer to the excellent survey~\cite{BurerL12} by Burer and Letchford for an overview of approaches for tackling mixed-integer nonlinear optimization problems in general.
One of these strategies is to introduce auxiliary variables for intermediate nonlinear terms~\cite{McCormick76}.
In our case, a straight-forward linearization is to introduce a variable $z_I$ for every subset $I$ of variables that appears in any of these polynomials, which yields the equivalent problem
\begin{subequations}
  \label{model_multilinear_linearized}
  \begin{alignat}{7}
    & \text{min }
      & \sum_{I \in \I_0} c^0_I z_I \label{model_multilinear_linearized_obj} \\
    & \text{s.t. }
      & \sum_{I \in \I_j} c^j_I z_I &\leq b_j &\quad& \forall j \in \{1,2,\dotsc,m\} \label{model_multilinear_linearized_ineqs} \\
    & & z_I &= \prod_{v \in I} x_v &\quad& \forall I \in \E \label{model_multilinear_linearized_prod} \\
    & & x_v &\in [\ell_v, u_v] &\quad& \forall v \in V, \label{model_multilinear_linearized_domain}
  \end{alignat}
\end{subequations}
where $\E \coloneqq \bigcup_{j=0}^m \I_j$ denotes the union of all involved variable subsets.
In the unconstrained case ($m=0$) it is known well~\cite{TawarmalaniS02} that there there exists an optimal solution in which each~$x_v$ is at its bound, that is, $x_v \in \{\ell_v, u_v\}$ holds for all $v \in V$.
Hence, by an affine transformation we can replace~\eqref{model_multilinear_linearized_domain} by $x_v \in \{0,1\}$ for all $v \in V$.
We now focus on linear relaxations for constraints~\eqref{model_multilinear_linearized_prod} and the requirement that $x$ is binary, whose set of feasible solutions is called the \emph{multilinear set}.
It is parameterized by the pair $G = (V, \E)$, which we can be interpreted as a hypergraph whose nodes index the original $x$-variables and whose hyperedges correspond to the (multlilinear) product terms~\cite{DelPiaK17}.
We will not rely on hypergraph concepts, but merely use these to visualize instances.
However, many previous results exploit hypergraph structures such as various cycle concepts~\cite{DelPiaK18,DelPiaK21,DelPiaD21,DelPiaW22,DelPiaW23}.

Every hypergraph $G = (V, \E)$ gives rise to a \emph{multilinear polytope}, defined as the convex hull
\begin{gather*}
  \ML(G) \coloneqq \conv \{ (x,z) \in \{0,1\}^V \times \{0,1\}^{\E} \mid z_I = \prod_{v \in I} x_v ~\forall I \in \E \}.
\end{gather*}
of the multilinear set.
Its simplest polyhedral relaxation is the \emph{standard relaxation}
\begin{subequations}
  \label{model_standard}
  \begin{alignat}{7}
    z_I &\leq x_v &\quad& \forall v \in I \in \E \label{model_standard_short} \\
    z_I + \sum_{v \in I} (1-x_v) &\geq 1 &\quad& \forall I \in \E \label{model_standard_long} \\
    z_I &\geq 0 &\quad& \forall I \in \E \label{model_standard_domain_edge} \\
    x_v &\in [0,1] &\quad& \forall v \in V, \label{model_standard_domain_node}
  \end{alignat}
\end{subequations}
which dates back to Fortet~\cite{Fortet60a,Fortet60b} and Glover and Wolsey~\cite{GloverW73,GloverW74}.
Unfortunately, this relaxation is often very weak~\cite{LuedtkeEtAl2012}.

In this paper we relate two strengthing techniques to each other: the first is by augmenting~\eqref{model_standard} with \emph{(extended) flower inequalities}~\cite{DelPiaK18,Khajavirad23}, which are additional inequalities valid for $\ML(G)$.
We will define the inequalities and the tightened relaxation in \cref{sec_flower}.
The second technique works by not linearizing each variable $z_I = \prod_{v \in I} x_v$ independently, but by using auxiliary variables for two disjoint subsets $I_1,I_2 \subseteq I$ with $I = I_1 \cup I_2$ instead, and inequalities similar to~\eqref{model_standard_short} and~\eqref{model_standard_long} for the product $z_I = z_{I_1} \cdot z_{I_2}$.
While this alone does not strengthen the relaxation, it is well known that it does so in case these variables $z_{I_1}, z_{I_2}$ appears in several linearization steps.
We will provide a concise definition of such a \emph{recursive McCormick linearization} in \cref{sec_mccormick}.

In a recent paper~\cite{Khajavirad23}, Khajavirad showed that extended flower inequalities dominate recursive McCormick linearizations in the sense that the latter can never give stronger linear programming bounds than the former.
Our main contribution is presented in \cref{sec_equivalent}, and states that the converse statement holds in some sense as well: one may need to intersect several McCormick linearizations with each other to achieve the same strength as that of extended flower inequalities.
This may be considered a big drawback, but it is worth noting that each McCormick linearization induces only a polynomial number of additional variables and constraints, while there exist exponentially many flower inequalities whose separation problem is NP-hard~\cite{DelPiaKS20}.
In fact, our result applies to linearizations that are more general than recursive McCormick linearizations, and we provide a simpler proof of Khajavirad's main result.
In \cref{sec_conclusions} we conclude the paper with the observation that also computationally both approaches are equivalent, and highlight an interesting open problem.

\section{Flower relaxation}
\label{sec_flower}

To simplify our notation we will write $z_{\{v\}} \coloneqq x_v$ for all $v \in V$, and denote by $\singletons \coloneqq \{ \{v\} \mid v \in V \}$ the corresponding set of singletons.
Note that this may in principle lead to an ambiguity in case $\{v\}$ is also part of $\E$.
However, already the standard relaxation~\eqref{model_standard} would imply $z_{\{v\}} = x_v$ in this case.
Hence, we from now on assume that $|I| \geq 2$ holds for all $I \in \E$.
Let us now define extended flower inequalities~\cite{Khajavirad23}, which are a generalization of flower inequalities~\cite{DelPiaK18}.

\begin{definition}[extended flower inequalities, extended flower relaxation]
  Let $I \in \E$ and let $J_1,\dotsc,J_k \in \E \cup \singletons$ be such that $J_1 \cup J_2 \cup \dotsb \cup J_k \supseteq I$ and $J_i \cap I \neq \emptyset$ holds for $i=1,2,\dotsc,k$.
  The \emph{extended flower inequality centered at~$I$ with neighbors $J_1,J_2,\dotsc,J_k$} is the inequality
  \begin{alignat}{7}
    z_I + \sum_{i=1}^k (1-z_{J_i}) \geq 1. \label{eq_flower}
  \end{alignat}
  The \emph{extended flower relaxation} $\flowerRel(G) \subseteq \R^{\E \cup \singletons}$ is defined as the intersection of $[0,1]^{\E \cup \singletons}$ with all extended flower inequalities.
\end{definition}

Note that an extended flower inequality centered at~$I$ with neighbors $J_1,J_2,\dotsc,J_k$ is called a \emph{flower inequality} if $J_i \cap J_j \cap I = \emptyset$ holds for all distinct $i,j$, that is, if the intersections $(J_1 \cap I), (J_2 \cap I), \dotsc, (J_k \cap I)$ form a partition of $I$.
The following proposition establishes validity of extended flower inequalities for the multilinear set.

\begin{proposition}
  For each hypergraph $G = (V,\E)$ we have $\ML(G) \subseteq \flowerRel(G)$.
\end{proposition}

\begin{proof}
  Let $z^\star \in \ML(G) \cap \{0,1\}^{\singletons \cup \E}$ be a vertex of $\ML(G)$ and consider an extended flower inequality centered at $I \in \E$ with neighbors $J_1,J_2,\dotsc,J_k \in \E \cup \singletons$.
  If $\sum_{i=1}^k (1-z^\star_{J_i}) \geq 1$, then the extended flower inequality is clearly satisfied.
  Otherwise, $\sum_{i=1}^k (1-z^\star_{J_i}) = 0$ implies $z^\star_{J_i} = 1$ for all $i \in \{1,2,\dotsc,k\}$.
  From this and $J_1 \cup J_2 \cup \dotsb \cup J_k \supseteq I$ we obtain $z^\star_{\{v\}} = 1$ for all $v \in I$, which implies $z^\star_I = \prod_{v \in I} z^\star_{\{v\}} = 1$ since $z^\star \in \ML(G)$.
  This shows that the left-hand side of~\eqref{eq_flower} is at least $1$, which concludes the proof.
\end{proof}

A key property of the extended flower relaxation is that it is compatible with projections in the sense that the (orthogonal) projection of the extended flower relaxation is the flower relaxation of the corresponding subgraph:

\begin{lemma}
  \label{thm_flower_project}
  Let $G = (V,\E)$ be a hypergraph, let $\E' \coloneqq \E \setminus \{I^\star\}$ be the set of hyperedges with $I^\star \in \E$ removed, and let $G' = (V, \E')$ be the corresponding hypergraph.
  Then the projection of $\flowerRel(G)$ onto the $z_I$-variables for all $I \in \E' \cup \singletons$ is equal to $\flowerRel(G')$.
\end{lemma}

\begin{proof}
  Let $P \subseteq \R^{\E' \cup \singletons}$ denote the projection of $\flowerRel(G)$ onto all variables but $z_{I^\star}$.
  Clearly, every extended flower inequality~\eqref{eq_flower} that does not involve $z_{I^\star}$ is present in both relaxations.
  This already shows $P \subseteq \flowerRel(G')$.

  For the reverse direction we apply Fourier-Motzkin elimination~\cite{Dines19,Fourier27,Motzkin36} to $\flowerRel(G)$ and $z_{I^\star}$.
  To this end, we need to consider pairs of inequalities~\eqref{eq_flower} in which $z_{I^\star}$ appears with opposite signs.
  Such a pair consists of one extended flower inequality centered at $I^\star$ with neighbors $H_1,H_2,\dotsc,H_k \in \E \cup \singletons$ and one extended flower inequality centered at another edge~$J$ with neighbors $I^\star, K_1, K_2, \dotsc, K_\ell \in \E \cup \singletons$.
  The sum of the two inequalities reads
  \begin{gather}
    z_{I^\star} + \sum_{i=1}^k (1-z_{H_i}) + z_J + (1-z_{I^\star}) + \sum_{i=1}^\ell (1-z_{K_i}) \geq 1 + 1
    \label{eq_fourier_motzkin1}
  \end{gather}
  We can assume that the $H_1, H_2, \dotsc, H_k$ are ordered such that $H_i \cap J \neq \emptyset$ holds if and only if $i \leq k'$ holds, where $k' \in \{0,1,\dotsc,k\}$ is a suitable index.
  We can rewrite~\eqref{eq_fourier_motzkin1} as
  \begin{equation*}
    z_J + \sum_{i=1}^{k'} (1-z_{H_i}) + \sum_{i=1}^\ell (1-z_{K_i}) \qquad+\qquad \sum_{i=k'+1}^k (1-z_{H_i}) \geq 1
  \end{equation*}
  which is the sum of the extended flower inequality~\eqref{eq_flower} centered at~$J$ with neighbors $H_1, H_2, \dotsc, H_{k'}, K_1, K_2, \dotsc, K_\ell$ and the bound inequalities $1-z_{H_i} \geq 0$ for all $i \in \{k'+1,k'+2,\dotsc,k\}$.
  By construction and by the choice of~$k'$ the edges (or singletons) $H_1,H_2,\dotsc,H_{k'},K_1,K_2,\dotsc,K_\ell$ indeed cover $J$.
  Hence, the combined inequality is implied by $\flowerRel(G')$, which establishes $P \supseteq \flowerRel(G')$.
  This concludes the proof.
\end{proof}

\section{Recursive linearizations and their relaxations}
\label{sec_mccormick}

For a digraph and a node $w$ we denote by $\outNeighbors(w)$ and $\inNeighbors(w)$ the \emph{successors} and \emph{predecessors} of $w$, i.e., the sets of nodes $u$ for which there is an arc from $w$ to $u$ and from $u$ to $w$, respectively.

\begin{definition}[recursive linearizations]
  Let~$V$ be a finite ground set.
  A \emph{recursive linearization} is a simple digraph $\D = (\V,\A)$ with $\singletons \subseteq \V \subseteq 2^V \setminus \{ \emptyset \}$ that satisfies $I \supseteq J$ for each arc $(I,J) \in \A$ as well as $\bigcup_{J \in \outNeighbors(I)} J = I$ for each $I \in \V \setminus \singletons$.
  We say that~$\D$ is a recursive linearization \emph{of a hypergraph $G = (V, \E)$} if $\E \subseteq \V$ holds and if each set $I \in \V$ with $\inNeighbors(I) = \emptyset$ belongs to $\E \cup \singletons$.
  Such a linearization is called \emph{partitioning} if for every node $I \in \V \setminus \singletons$ all its successors $J,J' \in \outNeighbors(I)$ (with $J \neq J'$) are disjoint, i.e., $J \cap J' = \emptyset$ holds.
  It is called \emph{binary} if $|\outNeighbors(I)| = 2$ holds for each $I \in \V \setminus \singletons$.
  A recursive linearization that is both, partitioning and binary, is called \emph{recursive McCormick linearization}.
\end{definition}

A recursive linearization encodes how each variable $z_I$ with $I \in \E$ is linearized by means of other variables.
First, such a variable~$z_I$ is created for each node $I \in \V$, where $z_{\{v\}}$ is the same as the variable~$x_v$ for each $v \in V$.
It is supposed to encode the product of the~$x_v$ for all $v \in V$, that is, $z_I = \prod_{v \in I} x_v$.
This will actually be done by encoding (using linear inequalities) that $z_I = \prod_{J \in \outNeighbors(I)} z_J$ holds for each $I \in \V$.
By induction we can assume that $z_J = \prod_{v \in J} x_v$ holds, which yields $z_I = \prod_{J \in \outNeighbors(I)} \prod_{v \in J} x_v$.
Since~$I$ is the union of all sets $\outNeighbors(I)$, a variable~$x_v$ appears in this product if and only if $v \in I$ holds.
Unless the linearization is partitioning, the variable may appear multiple times, but this does not harm since $x_v^2 = x_v$ holds for $x_v \in \{0,1\}$.
Notice that a linearization is partitioning if and only if, for every node $I \in \V$ and every $v \in I$, there is a unique path from~$I$ to~$\{v\}$.
\Cref{fig_linear} shows three different linearizations of the same hypergraph.

\begin{figure}[htpb]
  \subfloat[%
    A multilinear optimization problem with its hypergraph representation.
    Nodes are indicated as squares and hyperedges via ellipses.
    \label{fig_linear_problem}
  ]{%
    \begin{minipage}{0.48\textwidth}
      Multilinear optimization problem:
      \[
        \min \textcolor{red}{x_1 x_2 x_3} + \textcolor{green!60!black}{x_2x_3x_4} + \textcolor{purple}{x_1 x_2} \text{ s.t. } x \in \{0,1\}^4
      \]
      Hypergraph:
      \begin{center}
        \begin{tikzpicture}[
          vertex/.style={draw=black, very thick},
          ]
          \node[vertex] (1) at (-1.1,0) {1};
          \node[vertex] (2) at (0.3,0) {2};
          \node[vertex] (3) at (1.5,0) {3};
          \node[vertex] (4) at (3.3,0) {4};

          \draw[red, very thick,densely dotted] (0,0) ellipse (25mm and 10mm);
          \draw[green!60!black, very thick,dashed] (2,0) ellipse (25mm and 10mm);
          \draw[purple, very thick] (-0.5,0) ellipse (13mm and 13mm);

          \node[red] at (-2.2,-1.0) {$\{1, 2, 3\}$};
          \node[green!60!black] at (4.2,-1.0) {$\{2, 3, 4\}$};
          \node[purple] at (-1.8,1.3) {$\{1, 2\}$};
        \end{tikzpicture}
      \end{center}
    \end{minipage}
  }%
  \hfill%
  \subfloat[%
    The standard linearization $\D^{\text{(b)}}$ with $z_{\{1,2,3\}} = z_{\{1\}} \cdot z_{\{2\}} \cdot z_{\{3\}}$, $z_{\{2,3,4\}} = z_{\{2\}} \cdot z_{\{3\}} \cdot z_{\{4\}}$, $z_{\{1,2\}} = z_{\{1\}} \cdot z_{\{2\}}$ and $z_{\{2,3\}} = z_{\{2\}} \cdot z_{\{3\}}$, which is non-binary but partitioning.
    \label{fig_linear_standard}
  ]{%
    \begin{minipage}{0.48\textwidth}
      \begin{center}
        \begin{tikzpicture}[
          scale=0.9,
          node/.style={draw=black, thick, rounded corners},
          arc/.style={draw=black, very thick, ->},
        ]
          \node[node, red] (123) at (-3.5,0.0) {$\{1,2,3\}$};
          \node[node, green!60!black] (234) at (-3.5,-3.0) {$\{2,3,4\}$};
          \node[node, purple] (12) at (-2.5,-1.5) {$\{1,2\}$};
          \node[node] (1) at (0,0) {$\{1\}$};
          \node[node] (2) at (0,-1) {$\{2\}$};
          \node[node] (3) at (0,-2) {$\{3\}$};
          \node[node] (4) at (0,-3) {$\{4\}$};
    
          \foreach \u/\v in {123/1, 123/2, 123/3, 234/2, 234/3, 234/4, 12/1, 12/2}{%
            \draw[arc] (\u) -- (\v);
          }
        \end{tikzpicture}
      \end{center}
    \end{minipage}
  }%
  \\[5mm]
  \subfloat[%
    A recursive McCormick linearization $\D^{\text{(c)}}$ with $z_{\{1,2,3\}} = z_{\{1\}} \cdot z_{\{2,3\}}$, $z_{\{2,3,4\}} = z_{\{2,3\}} \cdot z_{\{4\}}$, $z_{\{1,2\}} = z_{\{1\}} \cdot z_{\{2\}}$ and $z_{\{2,3\}} = z_{\{2\}} \cdot z_{\{3\}}$.
    \label{fig_linear_mccormick}
  ]{%
    \begin{minipage}{0.48\textwidth}
      \begin{center}
        \begin{tikzpicture}[
          scale=0.8,
          node/.style={draw=black, thick, rounded corners},
          arc/.style={draw=black, very thick, ->},
        ]
          \node[node, red] (123) at (-4.5,0.0) {$\{1,2,3\}$};
          \node[node, green!60!black] (234) at (-4.5,-3.0) {$\{2,3,4\}$};
          \node[node, purple] (12) at (-2.1,-0.5) {$\{1,2\}$};
          \node[node, blue] (23) at (-2.5,-1.7) {$\{2,3\}$};
          \node[node] (1) at (0,0) {$\{1\}$};
          \node[node] (2) at (0,-1) {$\{2\}$};
          \node[node] (3) at (0,-2) {$\{3\}$};
          \node[node] (4) at (0,-3) {$\{4\}$};
    
          \foreach \u/\v in {123/1, 123/23, 234/23, 234/4, 12/1, 12/2, 23/2, 23/3}{%
            \draw[arc] (\u) -- (\v);
          }
        \end{tikzpicture}
      \end{center}
    \end{minipage}
  }%
  \hfill%
  \subfloat[%
    A binary non-partitioning linearization $\D^{\text{(d)}}$ with $z_{\{1,2,3\}} = z_{\{1,2\}} \cdot z_{\{2,3\}}$, $z_{\{2,3,4\}} = z_{\{2,3\}} \cdot z_{\{4\}}$, $z_{\{1,2\}} = z_{\{1\}} \cdot z_{\{2\}}$ and $z_{\{2,3\}} = z_{\{2\}} \cdot z_{\{3\}}$.
    \label{fig_linear_overlap}
  ]{%
    \begin{minipage}{0.48\textwidth}
      \begin{center}
        \begin{tikzpicture}[
          scale=0.8,
          node/.style={draw=black, thick, rounded corners},
          arc/.style={draw=black, very thick, ->},
        ]
          \node[node, red] (123) at (-4.5,0.0) {$\{1,2,3\}$};
          \node[node, green!60!black] (234) at (-4.5,-3.0) {$\{2,3,4\}$};
          \node[node, purple] (12) at (-2.1,-0.5) {$\{1,2\}$};
          \node[node, blue] (23) at (-2.5,-1.7) {$\{2,3\}$};
          \node[node] (1) at (0,0) {$\{1\}$};
          \node[node] (2) at (0,-1) {$\{2\}$};
          \node[node] (3) at (0,-2) {$\{3\}$};
          \node[node] (4) at (0,-3) {$\{4\}$};
    
          \foreach \u/\v in {123/12, 123/23, 234/23, 234/4, 12/1, 12/2, 23/2, 23/3}{%
            \draw[arc] (\u) -- (\v);
          }
        \end{tikzpicture}
      \end{center}
    \end{minipage}
  }%
  \caption{%
    Three linearizations for the minimization problem in \cref{fig_linear_problem} with the depicted hypergraph $G = (V, \E)$ with $V = \{1,2,3,4\}$ and
    $\E = \{
    \color{red}{\{1,2,3\}},
    \color{green!60!black}{\{2,3,4\}},
    \color{purple}{\{1,2\}},
    \color{blue}{\{2,3\}}
    \}$.
    The arcs that leave a node $I \subseteq V$ indicate the product that is used to represent~$z_I$. \\
    Consider the point $z^{(1)} \in \R^{\singletons cup \E}$ with $z^{(1)}_{\color{green!60!black}{\{2,3,4\}}} = 0$, $z^{(1)}_{\color{red}{\{1,2,3\}}} = z^{(1)}_{\color{purple}{\{1,2\}}} = z^{(1)}_{\{1\}} = z^{(1)}_{\{2\}} = z^{(1)}_{\{3\}} = \tfrac{1}{2}$ and $z^{(1)}_{\{4\}} = 1$.
    It is contained in $P(\D^{\text{(b)}})$.
    However, it is contained in neither $P_{\E}(\D^{\text{(c)}})$ nor in $P_{\E}(\D^{\text{(d)}})$ since in both linearizations the arc from $\color{red}{\{1,2,3\}}$ to $\color{blue}{\{2,3\}}$ implies $z^{(1)}_{\color{blue}{\{2,3\}}} \geq \tfrac{1}{2}$ and since $z^{(1)}_{\color{green!60!black}{\{2,3,4\}}} + (1 - z^{(1)}_{\color{blue}{\{2,3\}}}) + (1 - z^{(1)}_{\{4\}}) \geq 1$ implies $z^{(1)}_{\color{blue}{\{2,3\}}} \leq 0$. \\
    Consider the point $z^{(2)} \in \R^{\E \cup \singletons}$ with $z^{(2)}_{\color{red}{\{1,2,3\}}} = 0$, $z^{(2)}_{\color{green!60!black}{\{2,3,4\}}} = z^{(2)}_{\color{purple}{\{1,2\}}} = z^{(2)}_{\{2\}} = z^{(2)}_{\{3\}} = z^{(2)}_{\{4\}} = \tfrac{1}{2}$ and $z^{(2)}_{\{1\}} = 1$.
    It is contained in $P(\D^{\text{(b)}})$ and in $P_{\E}(\D^{\text{(d)}})$.
    However, it is not contained in $P_{\E}(\D^{\text{(c)}})$ since the arc from $\color{green!60!black}{\{2,3,4\}}$ to $\color{blue}{\{2,3\}}$ implies $z^{(2)}_{\color{blue}{\{2,3\}}} \geq \tfrac{1}{2}$ and since $z^{(2)}_{\color{red}{\{1,2,3\}}} + (1 - z^{(2)}_{\color{blue}{\{2,3\}}}) + (1 - z^{(2)}_{\{1\}}) \geq 1$ implies $z^{(2)}_{\color{blue}{\{2,3\}}} \leq 0$. \\
    Consider the point $z^{(3)} \in \R^{\E \cup \singletons}$ with $z^{(3)}_{\color{purple}{\{1,2\}}} = 0$ and $z^{(3)}_{\color{red}{\{1,2,3\}}} = z^{(3)}_{\color{green!60!black}{\{2,3,4\}}} = z^{(3)}_{\{1\}} = z^{(3)}_{\{2\}} = z^{(3)}_{\{3\}} = z^{(3)}_{\{4\}} = \tfrac{1}{2}$.
    It is contained in $P(\D^{\text{(b)}})$ and in $P_{\E}(\D^{\text{(c)}})$.
    However, it is not contained in $P_{\E}(\D^{\text{(d)}})$ since the arc from $\color{red}{\{1,2,3\}}$ to $\color{purple}{\{1,2\}}$ implies $\tfrac{1}{2} = z^{(3)}_{\color{red}{\{1,2,3\}}} \leq z^{(3)}_{\color{purple}{\{1,2\}}} = 0$.
  }
  \label{fig_linear}
\end{figure}

A recursive linearization states the interplay between the linearization variables.
The actual linear inequalities associated with a linearization are stated in the following definition.

\begin{definition}[relaxation, projected relaxation]
  Let $\D = (\V, \A)$ be a recursive linearization.
  Its \emph{relaxation} is the polyhedron $P(\D) \subseteq \R^{\V}$ defined by
  \begin{subequations}
    \label{model_linearization}
    \begin{alignat}{7}
      z_I &\leq z_J &\quad& \forall (I,J) \in \A \label{model_linearization_short} \\
      z_I + \sum_{J \in \outNeighbors(I)} (1-z_J) &\geq 1 &\quad& \forall I \in \V \setminus \singletons \label{model_linearization_long} \\
      z &\in [0,1]^\V. \label{model_linearization_domain}
    \end{alignat}
  \end{subequations}
  Its \emph{projected relaxation with respect to a given set $\targets \subseteq \V$} of \emph{target nodes} is the projection of $P(\D)$ onto the variables~$z_I$ for all $I \in \targets \cup \singletons$ and is denoted by $P_\targets(\D)$.
\end{definition}

Before we turn to our main result, we show that requiring a recursive linearization to be partitioning (resp.\ binary) is actually a restriction.
This requires the following simply lemma.

\begin{lemma}
  \label{thm_path_implied}
  Let $\D = (\V, \A)$ be a recursive linearization and let $I^\star,J^\star \in \V$ be two of its nodes.
  Then the inequality $z_{I^\star} \leq z_{J^\star}$ is valid for $P(\D)$ if and only if $\D$ contains a path from $I^\star$ to $J^\star$.
\end{lemma}

\begin{proof}
  Sufficiency follows by considering a path from $I^\star$ to $J^\star$ and adding inequalities~\eqref{model_linearization_short} for all its arcs.
  To see necessity, assume that $\D$ does not contain such a path.
  We define the vector $z \in \R^{\V}$ with
  \begin{gather*}
    z_I \coloneqq \begin{cases}
      0 & \text{if there exists a path from $I$ to $J^\star$,} \\
      \frac{1}{2} & \text{otherwise},
    \end{cases}
  \end{gather*}
  and claim that $z \in P(\D)$ holds.
  Inequality~\eqref{model_linearization_short} for an arc $(I,J) \in \A$ is clearly satisfied if $z_J = \frac{1}{2}$ holds.
  Otherwise, there exists a path from $J$ to $J^\star$ in $\D$, which can be extended to a path from $I$ to $J^\star$ using arc $(I,J)$.
  This shows that in this case $z_I = 0$ holds, and thus that the inequality is also satisfied.
  Similarly, inequality~\eqref{model_linearization_long} for a node $I$ is clearly satisfied since the definition of $z$ yields $1-z_J \geq 0.5$ for each successor $J \in \outNeighbors(I)$, and since $I$ has $|\outNeighbors(I)| \geq 2$ such successors.
  From $z \in P(\D)$ and $\frac{1}{2} = z_{I^\star} \not\leq z_{J^\star} = 0$ we conclude that $z_{I^\star} \leq z_{J^\star}$ is not valid for $P(\D)$.
\end{proof}

\begin{figure}[htpb]
  \subfloat[%
    A non-partitioning recursive McCormick linearization whose relaxation is not dominated by any partitioning recursive linearization.
    \label{fig_restriction_partitioning}
  ]{%
    \begin{minipage}{0.36\textwidth}
      \begin{center}
        \begin{tikzpicture}[
            scale=0.8,
              node/.style={draw=black, thick, rounded corners},
              arc/.style={draw=black, very thick, ->},
            ]
              \node[node, red] (123) at (-4.0,-1.0) {$\{1,2,3\}$};
              \node[node, green!60!black] (12) at (-2.0,-0.0) {$\{1,2\}$};
              \node[node, blue] (23) at (-2.0,-2.0) {$\{2,3\}$};
              \node[node] (1) at (0,-0) {$\{1\}$};
              \node[node] (2) at (0,-1) {$\{2\}$};
              \node[node] (3) at (0,-2) {$\{3\}$};
        
              \foreach \u/\v in {123/12, 123/23, 12/1, 12/2, 23/2, 23/3}{%
                \draw[arc] (\u) -- (\v);
              }
        \end{tikzpicture}
      \end{center}
    \end{minipage}
  }%
  \hfill
  \subfloat[%
    A non-binary recursive McCormick linearization whose relaxation is not dominated by any binary recursive linearization.
    \label{fig_restriction_binary}
  ]{%
    \begin{minipage}{0.60\textwidth}
      \begin{center}
        \begin{tikzpicture}[
            scale=0.8,
              node/.style={draw=black, thick, rounded corners},
              arc/.style={draw=black, very thick, ->},
            ]
              \node[node, red] (123456) at (-5.0,-2.5) {$\{1,2,3,4,5,6\}$};
              \node[node, green!60!black] (12) at (-2.0,-0.5) {$\{1,2\}$};
              \node[node, green!60!black] (34) at (-2.0,-2.5) {$\{3,4\}$};
              \node[node, green!60!black] (56) at (-2.0,-4.5) {$\{5,6\}$};
              \node[node, blue] (1234) at (5.1,-0.0) {$\{1,2,3,4\}$};
              \node[node, blue] (1256) at (5.1,-2.5) {$\{1,2,5,6\}$};
              \node[node, blue] (3456) at (5.1,-5.0) {$\{3,4,5,6\}$};
              \node[node, purple] (13) at (2.5,+0.5) {$\{1,3\}$};
              \node[node, purple] (24) at (2.5,-0.7) {$\{2,4\}$};
              \node[node, purple] (15) at (2.5,-1.9) {$\{1,5\}$};
              \node[node, purple] (26) at (2.5,-3.1) {$\{2,6\}$};
              \node[node, purple] (35) at (2.5,-4.3) {$\{3,5\}$};
              \node[node, purple] (46) at (2.5,-5.5) {$\{4,6\}$};
              \node[node] (1) at (0,+0.5) {$\{1\}$};
              \node[node] (2) at (0,-0.7) {$\{2\}$};
              \node[node] (3) at (0,-1.9) {$\{3\}$};
              \node[node] (4) at (0,-3.1) {$\{4\}$};
              \node[node] (5) at (0,-4.3) {$\{5\}$};
              \node[node] (6) at (0,-5.5) {$\{6\}$};
        
              \foreach \u/\v in {123456/12, 123456/34, 123456/56, 12/1, 12/2, 34/3, 34/4, 56/5, 56/6, 1234/13, 1234/24, 1256/15, 1256/26, 3456/35, 3456/46, 13/1, 13/3, 24/2, 24/4, 15/1, 15/5, 26/2, 26/6, 35/3, 35/5, 46/4, 46/6}{%
                \draw[arc] (\u) -- (\v);
              }
        \end{tikzpicture}
      \end{center}
    \end{minipage}
  }%
  \caption{Recursive McCormick linearizations showing that being partitioning or binary is a restriction.}
  \label{fig_restriction}
\end{figure}

\begin{proposition}
  Let $\D^{\text{\eqref{fig_restriction_partitioning}}} = (\V^{\text{\eqref{fig_restriction_partitioning}}}, \A^{\text{\eqref{fig_restriction_partitioning}}})$ be the recursive non-partitioning linearization (for hypergraph $G = (V,\E)$) from \cref{fig_restriction_partitioning} and let $\D = (\V, \A)$ be any recursive linearization with $P_{\E}(\D) \subseteq P(\D^{\text{\eqref{fig_restriction_partitioning}}})$.
  Then $\D$ is also non-partitioning.
\end{proposition}

\begin{proof}
  Inequalities $z_{\color{red}{\{1,2,3\}}} \leq z_{\color{green!60!black}{\{1,2\}}}$ and $z_{\color{red}{\{1,2,3\}}} \leq z_{\color{blue}{\{2,3\}}}$ are valid for $P(\D^{\text{\eqref{fig_restriction_partitioning}}})$.
  By \cref{thm_path_implied}, this implies that $\D$ must contain a path from $\color{red}{\{1,2,3\}}$ to $\color{green!60!black}{\{1,2\}}$ and one from $\color{red}{\{1,2,3\}}$ to $\color{blue}{\{2,3\}}$.
  Due to the cardinalities of the involved sets these paths must actually be arcs, which implies ${\color{green!60!black}\{1,2\}}, {\color{blue}\{2,3\}} \in \outNeighbors({\color{red}\{1,2,3\}})$.
  We conclude that $\D$ is not partitioning.
\end{proof}

\begin{proposition}
  Let $\D^{\text{\eqref{fig_restriction_binary}}} = (\V^{\text{\eqref{fig_restriction_binary}}}, \A^{\text{\eqref{fig_restriction_binary}}})$ be the recursive non-partitioning linearization (for hypergraph $G = (V,\E)$) from \cref{fig_restriction_binary} and let $\D = (\V, \A)$ be any recursive linearization with $P_\E(\D) \subseteq P(\D^{\text{\eqref{fig_restriction_binary}}})$.
  Then $\D$ is also non-binary.
\end{proposition}

\begin{proof}
  We assume, for the sake of contradiction, that $\D$ is binary.

  First, observe that $z_{{\color{blue}\{1,2,3,4\}}} \leq z_{{\color{purple}\{1,3\}}}$ and $z_{{\color{blue}\{1,2,3,4\}}} \leq z_{{\color{purple}\{2,4\}}}$ are valid for $P(\D^{\text{\eqref{fig_restriction_binary}}})$.
  By \cref{thm_path_implied}, this implies that $\D$ must contain paths from ${\color{blue}\{1,2,3,4\}}$ to ${\color{purple}\{1,3\}}$ and to ${\color{purple}\{2,4\}}$.
  Due to the cardinalities of the involved sets, these paths must actually be arcs, which implies $\outNeighbors({\color{blue}\{1,2,3,4\}}) = \{ {\color{purple}\{1,3\}}, {\color{purple}\{2,4\}} \}$ since $\D$ is binary.

  Second, inequalities $z_{{\color{red}\{1,2,3,4,5,6\}}} \leq z_{{\color{green!60!black}\{1,2\}}}$, $z_{{\color{red}\{1,2,3,4,5,6\}}} \leq z_{{\color{green!60!black}\{3,4\}}}$ and $z_{{\color{red}\{1,2,3,4,5,6\}}} \leq z_{{\color{green!60!black}\{5,6\}}}$ are valid for $P(\D^{\text{\eqref{fig_restriction_binary}}})$.
  Again by \cref{thm_path_implied}, this implies that $\D$ must contain paths from ${\color{red}\{1,2,3,4,5,6\}}$ to ${\color{green!60!black}\{1,2\}}$, ${\color{green!60!black}\{3,4\}}$ and to ${\color{green!60!black}\{5,6\}}$.
  Since $\D$ is binary, two such paths must have a common second node (after ${\color{red}\{1,2,3,4,5,6\}}$).
  By symmetry we can assume that these are the paths to ${\color{green!60!black}\{1,2\}}$ and to ${\color{green!60!black}\{3,4\}}$.
  Hence, this common node must be ${\color{blue}\{1,2,3,4\}}$.
  Due to $\outNeighbors({\color{blue}\{1,2,3,4\}}) = \{ {\color{purple}\{1,3\}}, {\color{purple}\{2,4\}} \}$, this contradicts the presence of the three paths from ${\color{red}\{1,2,3,4,5,6\}}$ to ${\color{green!60!black}\{1,2\}}$, ${\color{green!60!black}\{3,4\}}$ and to ${\color{green!60!black}\{5,6\}}$.
  We conclude that $\D$ is not binary.
\end{proof}

This shows that in principle a single non-binary or non-partitining linearization may be more powerful than any single recursive McCormick linearization.
However, in the next section we show that combinations of several McCormick linearizations are as powerful as the general recursive linearization.

\section{Comparison of projected relaxations}
\label{sec_equivalent}

\begin{theorem}
  \label{thm_equivalent}
  Let $G = (V,\E)$ be a hypergraph.
  Then the following polyhedra are equal:
  \begin{enumerate}[label=(\roman*)]
  \item
    \label{thm_equivalent_flower}
    flower relaxation $\flowerRel(G)$;
  \item
    \label{thm_equivalent_all}
    intersection of the projected relaxations $P_\E(\D)$ for all recursive linearizations~$\D$ of~$G$;
  \item
    \label{thm_equivalent_mccormick}
    intersection of the projected relaxations $P_{\E}(\D)$ for all recursive McCormick linearizations~$\D$ of $G$.
  \end{enumerate}
\end{theorem}

The fact that~\ref{thm_equivalent_flower} is contained in~\ref{thm_equivalent_mccormick} was established in~\cite{Khajavirad23}.
This was proved by applying Fourier-Motzkin elimination to the relaxation of a McCormick linearization, showing the all resulting inequalities are implied by extended flower inequalities.
In contrast to this, our proof works via the projection of the extended flower relaxation and is surprisingly simple.

\begin{proof}[Proof that \ref{thm_equivalent_flower} is contained in \ref{thm_equivalent_all}]
  Let $\D = (\V,\A)$ be a recursive linearization of a hypergraph $G = (V,\E)$.
  We need to show $\flowerRel(G) \subseteq P_{\E}(\D)$.
  To this end, we consider the hypergraph $G' = (V, \E')$ with $\E' \coloneqq \V \setminus \singletons$ and claim that $\flowerRel(G') \subseteq P(\D)$ holds.
  Note that $\E' \supseteq \E$ holds, that is, $G'$ contains all of $G$'s hyperedges.

  First, for any arc $(I,J) \in \A$, \eqref{model_linearization_short} is equivalent to the extended flower inequality~\eqref{eq_flower} centered at~$J$ with neighbor~$I$.
  Second, for any node $I \in \V \setminus \singletons$, \eqref{model_linearization_long} is equivalent to the extended flower inequality~\eqref{eq_flower} centered at~$I$ with neighbors $\outNeighbors(I)$.
  Third, for any $v \in \V$, also $0 \leq z_v \leq 1$ follows from the definition of $\flowerRel(G')$.
  This proves $\flowerRel(G') \subseteq P(\D)$.

  For the projection~$P$ of $\flowerRel(G')$ and for that of $P(\D)$, both onto the variables~$z_I$ for $I \in \E \cup \singletons$, we thus obtain $P \subseteq P_{\E}(\D)$.
  Successive application of \cref{thm_flower_project} to $P(\D)$ for every edge from $\E' \setminus \E$ yields $P = \flowerRel(G)$.
  Hence, $\flowerRel(G) \subseteq P_{\E}(\D)$ holds.
  Since $\flowerRel(G)$ is contained in each such projected relaxation, it is also contained in their intersection, which concludes the proof.
\end{proof}

\begin{proof}[Proof that \ref{thm_equivalent_all} is contained in \ref{thm_equivalent_mccormick}]
This holds since every recursive McCormick linearization is a recursive linearization.
\end{proof}

\DeclareDocumentCommand\U{}{\mathcal{U}}

\begin{proof}[Proof that \ref{thm_equivalent_mccormick} is contained in \ref{thm_equivalent_flower}]
It suffices to show that for any hypergraph $G = (V,\E)$ and any non-redundant extended flower inequality~\eqref{eq_flower} there exists a recursive binary partitioning linearization $\D = (\V, \A)$ whose projected relaxation with respect to~$\E$ implies this inequality.

To this end, consider an inequality~\eqref{eq_flower} centered at $I^\star$ with neighbors $J^\star_1, J^\star_2, \dotsc, J^\star_k$. 
Suppose there exists a neighbor $J^\star_{i^\star}$ that is not required for all (other) neighbors to cover $I^\star$.
Now consider the extended flower inequality
\begin{gather*}
  z_{I^\star} + (1-z_{J^\star_1}) + (1-z_{J^\star_2}) + \dotsb + (1-z_{J^\star_{i^\star-1}}) + (1-z_{J^\star_{i^\star+1}}) + \dotsb + (1-z_{J^\star_{k-1}}) + (1-z_{J^\star_k}) \geq 1
\end{gather*}
in which neighbor $J^\star_{i^\star}$ was removed.
Adding $1-z_{J^\star_{i^\star}} \geq 0$ to it yields our considered inequality, which shows that the latter is redundant.
Hence, from now on we assume that
\begin{gather}
  \text{for each $i \in \{1,2,\dotsc,k\}$ there exists a node $v \in I^\star \cap J^\star_i$ that belongs to no other neighbor.} \label{eq_invariant}
\end{gather}
We partition $I^\star$ into the sets
\begin{gather}
  L_i \coloneqq I^\star \cap \left(J^\star_i \setminus \bigcup_{j=1}^{i-1} J^\star_j \right) \quad i=1,2,\dotsc,k. \label{eq_construction_partition}
\end{gather}
Note that $L_{i^\star} \neq \emptyset$ holds by~\eqref{eq_invariant}.
We construct our linearization $\D = (\V,\A)$ as follows.
First, start with $\V \coloneqq \{ I^\star, L_1, L_2, \dotsc, L_k \} \cup \singletons$.
Second, for $i=k,k-1,\dotsc,3,2$, add the node $L_1 \cup L_2 \cup \dotsb \cup L_{i-1}$ to~$\V$ and the arcs from $L_1 \cup L_2 \cup \dotsb \cup L_i$ to $L_1 \cup L_2 \cup \dotsb \cup L_{i-1}$ and to~$L_i$.
Note that this yields a binary tree with root $I^\star$ and leaves~$L_i$.
We will later add paths from these leaves to the singletons, and for this purpose initialize the set~$\U$ of unprocessed nodes as $\U \coloneqq \{L_1, L_2, \dotsc, L_k\}$.

We now extend~$\D$ to also include the neighbors $J^\star_i$, processing all $i \in \{1,2,\dotsc,k\}$ one by one.
Throughout this process all nodes $I \in \V$ will be subsets of $I^\star$ or of those $J^\star_i$ that were processed so far.
If $J^\star_i = L_i$, then there is nothing to do and we continue with the next iteration $i+1$.
Otherwise, \eqref{eq_invariant} implies $J^\star_i \notin \V$.
Hence, we add this node $J^\star_i$ and the node $J^\star_i \setminus L_i$ (unless it exists) to $\V$, add the arcs $(J^\star_i, L_i), (J^\star_i, J^\star_i \setminus L_i)$ to $\A$, and add $J^\star_i \setminus L_i$ to~$\U$.
An example of such an incomplete linearization is depicted in \cref{fig_construction}.

\begin{figure}[htpb]
  \subfloat[%
    Support hypergraph with a center edge $\textcolor{red}{I^\star}$ and four neighbors $\textcolor{orange}{J^\star_1}$, $\textcolor{blue}{J^\star_2}$, $\textcolor{green!60!black}{J^\star_3}$ and $\textcolor{purple}{J^\star_4}$ along with the four sets $L_i$ as defined in~\eqref{eq_construction_partition}.
    \label{fig_construction_hypergraph}
  ]{%
    \begin{minipage}{0.335\textwidth}
      \begin{center}
        \begin{tikzpicture}[
            vertex/.style={draw=black, very thick},
          ]
          \foreach \i/\x/\y in {6/0/0, 3/-1.3/0, 9/-1.4/-1, 10/-1.4/-1.6, 7/0.75/-1, 8/0.75/-1.6, 4/0.1/1, 5/0.7/1, 1/-1.7/1, 2/-1.1/1,%
                                11/0.0/2.0, 12/0.8/2.0, 13/0.4/-2.6, 14/-2.6/-1, 15/-2.6/-1.6}{%
            \node[vertex] (\i) at (\x,\y) {$\i$};
          }
          \draw[rounded corners, red, ultra thick, dashed] (-2.2,-2.2) rectangle (1.4,1.6);
          \draw[red, thick, dashed] (-2.2,1.2) to +(-0.6,0.4) node[anchor=south] {$I^\star$};
          \draw[orange, very thick, densely dotted] (-0.8,-0.3) rectangle (-2.0,1.4);
          \draw[orange, thick, densely dotted] (-1.4,1.4) to +(+0.15,+0.7) node[anchor=south] {$= L_1$};
          \draw[rounded corners, orange, very thick] (-0.7,-0.4) rectangle (-2.1,1.5);
          \draw[orange, thick] (-1.6,1.5) to +(-0.3,+0.6) node[anchor=south] {$J^\star_1$};
          \draw[blue, very thick, densely dotted] (-0.4,-0.3) rectangle (1.1,1.4);
          \draw[blue, thick, densely dotted] (1.1,1.1) to +(+0.7,+0.2) node[anchor=west] {$L_2$};
          \draw[rounded corners, blue, very thick] (-0.6,-0.4) rectangle (1.3,2.4);
          \draw[blue, thick] (1.3,2.1) to +(+0.4,+0.2) node[anchor=west] {$J^\star_2$};
          \draw[green!60!black, very thick, densely dotted] (-0.4,-0.6) rectangle (1.1,-2.0);
          \draw[green!60!black, thick, densely dotted] (1.1,-1.7) to +(+0.7,-0.2) node[anchor=west] {$L_3$};
          \draw[rounded corners, green!60!black, very thick] (-0.5,+0.4) rectangle (1.2,-3.0);
          \draw[green!60!black, thick] (1.2,-2.5) to +(+0.5,-0.2) node[anchor=west] {$J^\star_3$};
          \draw[purple, very thick, densely dotted] (-0.8,-0.5) rectangle (-2.0,-2.0);
          \draw[purple, thick, densely dotted] (-1.6,-2.0) to +(-0.2,-0.6) node[anchor=north] {$L_4$};
          \draw[rounded corners, purple, very thick] (0.4,0.5) rectangle (-3.0,-2.1);
          \draw[purple, thick] (-2.4,-2.1) to +(-0.2,-0.5) node[anchor=north] {$J^\star_4$};
        \end{tikzpicture}
      \end{center}
    \end{minipage}
  }%
  \hfill
  \subfloat[%
    Corresponding incomplete McCormick linearization from the proof.
    The unprocessed nodes $\U = \{ \textcolor{orange}{\{1,2,3\}}$, $\textcolor{blue}{\{4,5,6\}}$, $\textcolor{green!60!black}{\{7,8\}}$, $\textcolor{purple}{\{9,10\}}$, $\{11,12\}$, $\{13\}$, $\{14,15\} \}$ are depicted with a dashed border.
    These are pairwise disjoint and hence the completion of the linearization is arbitrary.
    \label{fig_construction_linearization}
  ]{%
    \begin{minipage}{0.62\textwidth}
      \begin{center}
        \begin{tikzpicture}[
            node/.style={draw=black, thick, rounded corners},
            arc/.style={draw=black, very thick, ->},
          ]
          \node[node, red] (1-2-3-4-5-6-7-8-9-10) at (0.0,0.0) {$\{1,2,3,4,5,6,7,8,9,10\}$};
          \node[node, black] (1-2-3-4-5-6-7-8) at (-1.0,-1.0) {$\{1,2,3,4,5,6,7,8\}$};
          \node[node, black] (1-2-3-4-5-6) at (-2.0,-2.0) {$\{1,2,3,4,5,6\}$};
          \node[node, orange, dashed] (1-2-3) at (-4.0,-3.0) {$\{1,2,3\}$};
          \node[node, blue, dashed] (4-5-6) at (-2.0,-3.0) {$\{4,5,6\}$};
          \node[node, green!60!black, dashed] (7-8) at (-0.25,-3.0) {$\{7,8\}$};
          \node[node, purple, dashed] (9-10) at (1.5,-3.0) {$\{9,10\}$};
          \node[node, blue] (4-5-6-11-12) at (-4.0,-1.0) {$\{4,5,6,11,12\}$};
          \node[node, green!60!black] (7-8-13) at (2.4,-1.0) {$\{7,8,13\}$};
          \node[node, purple] (9-10-14-15) at (3.8,-2.0) {$\{9,10,14,15\}$};

          \node[node, black, dashed] (11-12) at (-4.5,-4.5) {$\{11,12\}$};
          \node[node, black, dashed] (13) at (+2.4,-4.5) {$\{13\}$};
          \node[node, black, dashed] (14-15) at (+3.8,-4.5) {$\{14,15\}$};

          \draw[arc] (1-2-3-4-5-6-7-8-9-10) to (1-2-3-4-5-6-7-8);
          \draw[arc] (1-2-3-4-5-6-7-8-9-10) to[bend left] (9-10);
          \draw[arc] (1-2-3-4-5-6-7-8) to (1-2-3-4-5-6);
          \draw[arc] (1-2-3-4-5-6-7-8) to[bend left] (7-8);
          \draw[arc] (1-2-3-4-5-6) to (1-2-3);
          \draw[arc] (1-2-3-4-5-6) to (4-5-6);
          \draw[arc] (4-5-6-11-12) to[bend right] (4-5-6);
          \draw[arc] (4-5-6-11-12) to[bend right] (11-12);
          \draw[arc] (7-8-13) to (7-8);
          \draw[arc] (7-8-13) to (13);
          \draw[arc] (9-10-14-15) to (9-10);
          \draw[arc] (9-10-14-15) to (14-15);
        \end{tikzpicture}
      \end{center}
    \end{minipage}
  }%
  \caption{Hyperedges for a flower inequality with four neighbors and the corresponding incomplete linearization from the proof that \ref{thm_equivalent_mccormick} is contained in \ref{thm_equivalent_flower}.}
  \label{fig_construction}
\end{figure}

After all $J^\star_i$ have been processed, we process all unprocessed nodes:
While $\U \neq \emptyset$, pick a node $I \in \U$.
If $I \in \V$, then we continue.
Otherwise, $|I| \geq 2$ must hold due to $\singletons \subseteq \V$.
Arbitrarily partition $I = I_1 \cup I_2$ with $I_1,I_2 \neq \emptyset$.
Add the nodes $I_1,I_2$ to~$\V$ (unless they exist) and the arcs $(I, I_1),(I, I_2)$ to~$\A$.
Remove~$I$ from~$\U$ and add~$I_1$ and~$I_2$ to it.
Note that this process terminates since the cardinalities of the nodes $I_1,I_2$ that are added to~$\U$ are smaller than that of the removed set $I$.

It is now easily verified that~$\D$ is a recursive linearization for~$G$ that is, by construction, binary and partitioning.
It has the following properties: $\D$ contains, for each $i \in \{1,2,\dotsc,k\}$ a path (of length~$0$ or~$1$) from $J^\star_i$ to~$L_i$ as well as a path from $I^\star$ to~$L_i$ (of length $k-i+1$ with the inner nodes $L_1 \cup L_2 \cup \dotsb \cup L_j$ for $j=k,k-1,\dotsc,i+1$ in that order).
We claim that the following inequalities are valid for $P(\D)$:
\begin{subequations}
  \label{eq_implied}
  \begin{alignat}{7}
    (1-z_{J^\star_i}) - (1-z_{L_i}) &\geq 0 &\quad& i = 1,2,\dotsc, k \label{eq_implied_short} \\
    z_{(L_1 \cup L_2 \cup \dotsb \cup L_j)} + (1-z_{(L_1 \cup L_2 \cup \dotsb \cup L_{j-1})}) + (1-z_{L_j}) & \geq 1 &\quad& j=2,3,\dotsc,k \label{eq_implied_long}
  \end{alignat}
\end{subequations}
Note that if $L_i = J^\star_i$, the two variables in~\eqref{eq_implied_short} are the same, and otherwise the inequality follows from~\eqref{model_linearization_short} due to the arc $(J^\star_i,L_i) \in \A$.
Inequality~\eqref{eq_implied_long} corresponds to~\eqref{model_linearization_long} for the two arcs that leave $L_1 \cup L_2 \cup \dotsb \cup L_j$.
The sum of all inequalities~\eqref{eq_implied_long} reads
\begin{gather*}
  z_{I^\star} + \sum_{i=1}^k (1-z_{L_i}) \geq 1.
\end{gather*}
Adding~\eqref{eq_implied_short} for $i=1,2,\dotsc,k$ yields the desired flower inequality~\eqref{eq_flower} centered at $I^\star$ with neighbors $J^\star_1$, $J^\star_2,\dotsc,J^\star_k$.
Clearly, this inequality is also implied by the projection $P_{\E}(\D)$ of $P(\D)$ onto the variables~$z_I$ for $I \in \E \cup \singletons$, which concludes the proof.
\end{proof}

\section{Discussion}
\label{sec_conclusions}

We would like to conclude our paper with the following consequences for computational complexity.
Suppose we have a family of hypergraphs $G = (V,\E)$ for which we can solve the separation problem for extended flower inequalities~\eqref{eq_flower} in polynomial time.
For instance, this is the case if the degree of the linearized polynomials is bounded by a constant, since then we need to consider only flowers whose number of neighbors is bounded by a constant.
The straight-forward way of using such a separation algorithm is to maintain a subset of generated extended flower inequalities and, whenever necessary, query the algorithm with a point $\hat{z} \in \R^{\singletons \cup \E}$ to check if $\hat{z} \in \flowerRel(G)$ holds or, if a violated inequality is determined, augment our subset.
Note that such a point $\hat{z}$ then typically satisfies the generated inequalities, say, in the context of the Ellipsoid method~\cite{Khachiyan79} or when doing row generation using the Simplex method~\cite{Dantzig51}.
As an alternative, we could now maintain a set of recursive linearizations.
For a point $\hat{z}$ we can now run the same algorithm, but apply our construction from the proof of \cref{thm_equivalent} (that \ref{thm_equivalent_mccormick} is contained in \ref{thm_equivalent_flower}) to obtain a new linearization that we add to our set.
The disadvantage is that adding such a new recursive linearization requires us to add new variables and new inequalities.
An advantage could be that a new recursive linearization can imply several flower inequalities at once.
This potential is illustrated in \cref{fig_mccormick_vs_flower}, where a quadratic number of flower inequalities is replaced by only one additional node in the recursive linearization.

\begin{figure}[htpb]
  \begin{center}
    \begin{tikzpicture}[
      yscale=0.70,
      vertex/.style={draw=black, very thick},
      ]
      \node[vertex] (u1) at (-0.33,0) {$u_1$};
      \node[vertex] (u2) at (0.33,0) {$u_2$};

      \foreach \k in {1, 2, 3, 4, 5, 6, 7, 8, 9, 10, 11, 12}{%
        \begin{scope}[rotate=30*\k]
          \node[vertex] (v-\k) at (3.0, 0) {$v_{\k}$};
          \draw[very thick] (15mm,0) ellipse (27mm and 10mm);
        \end{scope}
      }
    \end{tikzpicture}
  \end{center}
  \caption{%
    An instance in which several extended flower inequalities are captured by one recursive McCormick linearization.
    The hypergraph $G = (V,\E)$ has $V = \{ \{v_i, w_i\} \mid i = 1,2,\dotsc,k \} \cup \{ \{u_1, u_2\} \}$ and $\E = \{ \{ u_1, u_2, v_i, w_i \} \mid i = 1,2,\dotsc,k \}$ (for $k=12$).
    It has $k(k-1)$ non-redundant extended flower inequalities by considering each edge as the center edge and choosing any other edge as the unique neighbor.
    However, it admits a recursive McCormick linearization with one additional variable $z_{\{u_1, u_2\}}$ and $2k+3$ extra inequalities ($z_{\{u_1,u_2,v_i\}} \leq z_{\{u_1,u_2\}}$ and $z_{\{u_1,u_2,v_i\}} + (1-z_{\{u_1,u_2\}}) + (1-z_{v_i}) \geq 1$ for $i=1,2,\dotsc,k$ as well as those for $z_{\{u_1,u_2\}} = z_{u_1} \cdot z_{u_2}$), which turn $3k$ inequalities ($z_{\{u_1,u_2,v_i\}} \leq z_{u_1}$, $z_{\{u_1,u_2,v_i\}} \leq z_{u_2}$ and $z_{\{u_1,u_2,v_i\}} + (1-z_{u_1} + (1-z_{u_2}) + (1-z_{v_i}) \geq 1$ for $i=1,2,\dotsc,k$) from the standard relaxation redundant.
  }
  \label{fig_mccormick_vs_flower}
\end{figure}

Let us consider the intersection of two projected relaxations $P_{\E}(\D^{(1)}), P_{\E}(\D^{(2)})$ for linearizations $\D^{(j)} = (\V^{(j)}, \A^{(j)})$, $j=1,2$ for some hypergraph $G = (V,\E)$.
When working with recursive linearizations one would not carry out the projection, but actually maintain variables for all nodes in $\V^{(1)}$ and in $\V^{(2)}$.
The intersection is achieved by identifying the variables $z_I$ for all $I \in \E$.
However, if there is another variable $I \in (\V^{(1)} \cap \V^{(2)}) \setminus \E$, then it formally exists in both formulations, say as $z^{(1)}_I$ and $z^{(2)}_I$ and has the intended meaning of $z^{(j)}_I = \prod_{i \in I} x_i$ for $j=1,2$ in both linearizations.
Hence, we can identify these variables $z^{(1)}_I = z^{(2)}_I$ with each other.
First, this reduces the number of variables.
Second, if some arcs (that involve $I$) exist in both recursive linearizations, then this may also reduce the number of inequalities.
Notice that, after doing this, there might be multiple inequalities~\eqref{model_linearization_long} for a single $I \subseteq V$.
Third, this might strengthen the formulation.
In fact, \emph{running intersection inequalities}~\cite{DelPiaK21} can be seen as a strengthened version of flower inequalities.
We leave it as an open problem to investigate how this strengthening of intersected recursive linearizations compares to strengthening of flower inequalities to running intersection inequalities.

\bibliographystyle{plainurl}
\bibliography{mccormick-strikes-back}

\end{document}